\documentclass[10pt,reqno]{amsart}
\usepackage{amsmath,amssymb,mathtools}
\usepackage{microtype}
\usepackage{xcolor}
\usepackage{hyperref}
\hypersetup{colorlinks=true,linkcolor=black,citecolor=black,urlcolor=black}
\numberwithin{equation}{section}
\newtheorem{theorem}{Theorem}[section]
\newtheorem{proposition}[theorem]{Proposition}
\newtheorem{lemma}[theorem]{Lemma}

\newtheorem{definition}[theorem]{Definition}
\theoremstyle{remark}

\newcommand{\Zd}{\mathbb Z^d}
\newcommand{\R}{\mathbb R}

\newcommand{\D}{\Delta_{\rm d}}

\newcommand{\supp}{\operatorname{supp}}
\newcommand{\norm}[1]{\left\|#1\right\|}

\newcommand{\less}{\lesssim}
\title[Imaginary powers of discrete Laplacians]{Sharp estimates for imaginary powers of discrete Laplacians}

\author[T. A. Bui]{The Anh Bui}
\address{School of Mathematical and Physical Sciences, Macquarie University, NSW 2109, Australia}
\email{the.bui@mq.edu.au}

\author[T. Q. Bui]{The Quan Bui}
\address{University of Economics Ho Chi Minh City (UEH), Ho Chi Minh City, Vietnam}
\email{buithequan2002@yahoo.com}

\author[T. T. Cao]{Thanh Tinh Cao}
\address{Department of Mathematics and Physics, University of Information
	Technology, Ho Chi Minh City, Vietnam}	
\address{	Vietnam National
	University, Ho Chi Minh City, Vietnam}
\email{tinhct@uit.edu.vn}

\subjclass[2020]{Primary 42B15, 42B20; Secondary 42B30, 39A12}
\keywords{Discrete Laplacian, imaginary powers, Hardy spaces, sharp bounds, spectral multipliers}
\begin{document}
\begin{abstract}
Let $\Delta_{\rm d}$ be the nonnegative discrete Laplacian on $\mathbb Z^d$.
We prove that, for $0<p\le1$,
\[
 \|\Delta_{\rm d}^{i\gamma}\|_{H^p_{\Delta_{\rm d}}\to H^p_{\Delta_{\rm d}}}
 \lesssim (1+|\gamma|)^{d(1/p-1/2)},
\]
and, for $1<p<\infty$,
\[
 \|\Delta_{\rm d}^{i\gamma}\|_{\ell^p\to\ell^p}
 \lesssim (1+|\gamma|)^{d|1/p-1/2|}.
\]
The estimates are sharp with respect to the exponent of $1+|\gamma|$.
 We also explain why the general spectral
multiplier theorem and the usual Calder\'on--Zygmund approach do not give
the sharp power.
\end{abstract}
\maketitle

\section{Introduction}
Let $d\ge1$ and let $e_1,\ldots,e_d$ be the standard basis of $\R^d$.
We consider the nonnegative discrete Laplacian on $\Zd$ defined by
\begin{equation}\label{eq:L}
 \Delta_{\rm d} f(n)=\sum_{j=1}^d\bigl(2f(n)-f(n+e_j)-f(n-e_j)\bigr),
 \qquad n\in\Zd.
\end{equation}
The operator $\Delta_{\rm d}$ is nonnegative and self-adjoint on $\ell^2(\Zd)$.
Its symbol under the discrete Fourier transform is
\begin{equation}\label{eq:lambda}
 \lambda(\theta)=4\sum_{j=1}^d\sin^2(\theta_j/2),
 \qquad \theta\in[-\pi,\pi]^d.
\end{equation}
For $\gamma\in\R$, the imaginary power $\Delta_{\rm d}^{i\gamma}$ is defined
by the spectral theorem. The aim of this paper is to find the optimal
growth of its norms in terms of $1+|\gamma|$.

The imaginary powers of the Laplacian on $\R^d$ have been studied
extensively. See, for example, Sikora and Wright \cite{SW}.
{For Hardy spaces associated with operators, upper estimates
for imaginary powers were obtained under Gaussian heat kernel assumptions
in \cite{BBHH}. More recently, sharp estimates were studied under
Davies--Gaffney estimates in \cite{TBBDL}.}
It is natural to consider the corresponding problem for the discrete
Laplacian. However, the usual methods cannot be applied directly.
{Indeed, the heat semigroup $e^{-t\Delta_{\rm d}}$ does not satisfy
 the standard Gaussian upper bound or the usual Davies--Gaffney estimate.
 For example, if $n$ and $m$ are nearest neighbours, then
 $p_t(n,m)=t+O(t^2)$ as $t\to0^+$, which is incompatible with the
 Gaussian off-diagonal decay $C\exp(-c|n-m|^2/t)$.
 Thus the assumptions of \cite{BBHH,TBBDL} are not satisfied by
 the discrete Laplacian.}
The wave propagator also does not have finite
propagation speed with respect to the lattice distance. In addition,
the symbol \eqref{eq:lambda} is periodic, and the Euclidean Fourier
multiplier argument does not immediately apply.

A general spectral multiplier theorem was obtained in
\cite[Theorem~2.5]{BD}. That theorem requires an order of
smoothness greater than $d/2$ and does not directly give the sharp
power of $1+|\gamma|$. Even a near-endpoint version of the general
spectral multiplier argument, with an arbitrarily small loss in the
smoothness order, would give only a bound of the form
\begin{equation}\label{eq:epsilon-bound}
 \|\Delta_{\rm d}^{i\gamma}\|_{\ell^p\to\ell^p}
 \le C_{p,d,\varepsilon}(1+|\gamma|)^{d|1/p-1/2|+\varepsilon}.
\end{equation}
This does not provide the sharp power. 
The purpose of the present paper is to remove the loss in
\eqref{eq:epsilon-bound} and obtain the sharp estimates.
{We prove not only the upper bounds but also matching lower
bounds. In particular, the powers of $1+|\gamma|$ in both the Hardy-space
and $\ell^p$ estimates are optimal.}

For $0<p\le1$, let $H^p_{\Delta_{\rm d}}(\Zd)$ be the Hardy space associated
with $\Delta_{\rm d}$, as introduced in \cite[Section~3]{BD}. In fact, the Hardy spaces $H^p_{\Delta_{\rm d}}(\Zd)$ coincide with the classical Hardy spaces $H^p(\Zd)$ defined on $\Zd$; see Section~2. 
We recall their definition and the relevant properties in Section~2.
Our main results are as follows.

\begin{theorem}\label{thm:main}
	Let $d\ge1$ and $0<p\le1$. Then there exists $C_{p,d}>0$ such that for all
	$\gamma\in\mathbb R$ and $f\in H^p_{\Delta_{\rm d}}(\mathbb Z^d)$,
	\begin{equation}\label{eq:Hp}
		\|\Delta_{\rm d}^{i\gamma}f\|_{H^p_{\Delta_{\rm d}}(\mathbb Z^d)}
		\le C_{p,d}(1+|\gamma|)^{d(1/p-1/2)}
		\|f\|_{H^p_{\Delta_{\rm d}}(\mathbb Z^d)}.
	\end{equation}
	Moreover, the estimate \eqref{eq:Hp} is sharp with respect to the exponent of $1+|\gamma|$.
\end{theorem}

We also obtain the sharp estimates on $\ell^p(\mathbb Z^d)$.
\begin{theorem}\label{thm:lp}
	Let $d\ge1$ and $1<p<\infty$. Then there exists $C_{p,d}>0$ such that for all
	$\gamma\in\mathbb R$ and $f\in \ell^p(\mathbb Z^d)$,
	\begin{equation}\label{eq:Lp}
		\|\Delta_{\rm d}^{i\gamma}\|_{\ell^p(\mathbb Z^d)\to\ell^p(\mathbb Z^d)}
		\le C_{p,d}  (1+|\gamma|)^{d|1/p-1/2|},\qquad\gamma\in\mathbb R.
	\end{equation}
Moreover, the estimate \eqref{eq:Lp} is sharp with respect to the exponent of $1+|\gamma|$.
\end{theorem}

We emphasize that the estimates in Theorems~\ref{thm:main} and
\ref{thm:lp} are sharp. The main difficulty is to obtain the optimal
power in the upper bounds. Our argument is based on localized weighted
$\ell^2$ kernel estimates for spectral multipliers of $\Delta_{\rm d}$ and
on the molecular characterization of $H^p_{\Delta_{\rm d}}$. Differentiating the localised imaginary-power multiplier produces powers of \(1+|\gamma|\). We absorb these factors in the off-diagonal estimates by increasing the ball’s radius by a factor of \(1+|\gamma|\). The volume of the dilated ball gives the remaining factor and hence the exponent in Theorem~\ref{thm:main}. Interpolation and duality yield the upper
estimates in Theorem~\ref{thm:lp}.

The sharpness  for $1<p<\infty$ are obtained by stationary phase
on a fixed region of the torus. The sharpness in Theorem~\ref{thm:main}
then follows from the $\ell^p$ lower estimates and the complex
interpolation of Hardy spaces. We also discuss at the end of the paper
why the classical Calder\'on--Zygmund approach does not give the sharp
power of $1+|\gamma|$.

The paper is organized as follows. Section~2 contains the kernel estimate
and some facts about Hardy spaces. In Section~3 we prove the main results.
The final subsection explains the limitation of the classical
Calder\'on--Zygmund method.

\section{Preliminaries}\label{sec:prelim}
We write $B(n,r)=\{m\in\Zd:|m-n|<r\}$, and denote by $|B|$ the counting
measure of a ball $B$. We have $|B(n,r)|\simeq(1+r)^d$. For a ball
$B=B(n_B,r_B)$, put $S_0(B)=B$ and
$S_j(B)=2^jB\setminus2^{j-1}B$, $j\ge1$.

\subsection{A localized kernel estimate}
We begin with an estimate for the spectral multipliers of $\Delta_{\rm d}$.
It is an immediate consequence of the joint multiplier estimate in
\cite[Proposition~2.1]{BD}. We include the details to make the notation
and the dependence on the scale clear.

\begin{proposition}\label{prop:kernel}
Let $\psi\in C_c^\infty((1/4,4))$. Let $N$ be a nonnegative integer and $0<R\le4\sqrt d$. Suppose that
\[
 F_R(\lambda)=\psi(R^{-2}\lambda)G_R(\lambda),
\]
where $G_R$ is smooth on the support of the cutoff. Then
\begin{equation}\label{eq:W2}
 \sum_{n\in\Zd}|K_{F_R(\D)}(n,m)|^2(1+R|n-m|)^{2N}
 \le C_{N,d}R^d\|F_R(R^2\cdot)\|_{W_N^\infty(\R)}^2
\end{equation}
for all $m\in \Zd$, where $\|F\|_{W_N^\infty(\R)}=\sum_{j=0}^N\|F^{(j)}\|_\infty$.
The constant is independent of $R$ and $m$.
\end{proposition}
\begin{proof}
Let $\Delta_{{\rm d},j}$ be the $j$th coordinate discrete Laplacian and write
$\vec\Delta_{\rm d}=(\Delta_{{\rm d},1},\ldots,\Delta_{{\rm d},d})$. Set
$\widetilde F_R(\xi)=F_R(\xi_1+\cdots+\xi_d)$ for $\xi\in[0,\infty)^d$.
Then $F_R(\Delta_{\rm d})=\widetilde F_R(\vec\Delta_{\rm d})$.

Suppose first that $R\le\sqrt d$. Since $\supp F_R\subset[R^2/4,4R^2]$,
we have $\supp\widetilde F_R\subset[0,(2R)^2]^d$.
We may apply \cite[Proposition~2.1]{BD} with parameter $2R$. Since
$1+R|n-m|\simeq1+2R|n-m|$,
\[
\begin{aligned}
\sum_{n\in\Zd}|K_{F_R(\D)}(n,m)|^2(1+R|n-m|)^{2N} 
& \lesssim R^d
 \|\widetilde F_R((2R)^2\cdot)\|_{W_N^\infty(\R^d)}^2\\
& \lesssim R^d\|F_R(R^2\cdot)\|_{W_N^\infty(\R)}^2.
\end{aligned}
\]
Here we used the chain rule and the fact that
$\widetilde F_R((2R)^2\xi)=F_R(4R^2\sum_{j=1}^d\xi_j)$.

If $\sqrt d<R\le4\sqrt d$, then $R$ is bounded above and below by
positive constants depending on $d$. The convolution kernel of
$F_R(\D)$ is the inverse Fourier transform of
$F_R(\lambda(\theta))$, where $\lambda$ is given by
\eqref{eq:lambda}. Plancherel's theorem, followed by the chain rule,
gives
\[
 \sum_{n\in\Zd}|K_{F_R(\D)}(n,m)|^2(1+|n-m|)^{2N}
 \lesssim_{N,d}\sum_{j=0}^N\|F_R^{(j)}\|_\infty^2.
\]
Since $R\simeq_d1$, the right-hand side is bounded by the right-hand
side of \eqref{eq:W2}. 

This completes the proof.
\end{proof}

\subsection{Hardy spaces associated with $\Delta_{\rm d}$}
We recall some definitions and properties of the Hardy spaces
associated with $\Delta_{\rm d}$; see \cite[Section~3]{BD}.
Choose a real-valued function $\psi\in C_c^\infty((1/4,4))$
such that
\begin{equation}\label{eq-psi function}
 \sum_{\ell\in\mathbb Z}\psi(2^{-2\ell}\lambda)=1,
 \qquad \lambda>0.
\end{equation}
For $f\in\ell^2(\Zd)$, we define the discrete square function
\begin{equation}\label{eq:square}
 S_\psi f(n)=\left(\sum_{\ell\in\mathbb Z}
 |\psi(2^{-2\ell}\Delta_{\rm d})f(n)|^2\right)^{1/2},
 \qquad n\in\Zd.
\end{equation}
For $0<p<\infty$, the Hardy space $H^p_{\Delta_{\rm d}}(\Zd)$ is defined
to be the completion of the set
\[
 \{f\in\ell^2(\Zd): S_\psi f\in\ell^p(\Zd)\}
\]
under the quasi-norm
\[
 \|f\|_{H^p_{\Delta_{\rm d}}(\Zd)}=\|S_\psi f\|_{\ell^p(\Zd)}.
\]
This definition is equivalent to the area square-function definition
in \cite[Section~3]{BD}, by \cite[Theorem~3.4]{BD}. It is also
independent of the choice of an admissible cutoff, up to equivalent
quasi-norms. Moreover, by \cite[Theorem~3.5]{BD},
\[
 H^p_{\Delta_{\rm d}}(\Zd)=\ell^p(\Zd),\qquad 1<p<\infty,
\]
with equivalent norms.

We next recall the molecular characterization of these spaces from
\cite[Definitions~3.1--3.2 and Theorem~3.3]{BD}.
\begin{definition}\label{def:molecule}
Let $\varepsilon>0$, $0<p\le1$ and $M\in\mathbb N$.
A function $a$ is called a $(p,M,\varepsilon)_{\Delta_{\rm d}}$-molecule
associated with a ball $B\subset\Zd$ if there exists a function $b$
such that
\begin{enumerate}
 \item[(i)] $a=\Delta_{\rm d}^M b$;
 \item[(ii)] for every $k=0,1,\ldots,M$ and $j=0,1,2,\ldots$,
 \begin{equation}\label{eq:mol}
  \|\Delta_{\rm d}^k b\|_{\ell^2(S_j(B))}
  \le 2^{-j\varepsilon}r_B^{2(M-k)}|2^jB|^{1/2-1/p}.
 \end{equation}
\end{enumerate}
\end{definition}

\begin{definition}\label{def:mol-representation}
Let $\varepsilon>0$, $0<p\le1$ and $M\in\mathbb N$.
We say that
\[
 f=\sum_{j=0}^\infty\lambda_j a_j
\]
is a molecular $(p,M,\varepsilon)_{\Delta_{\rm d}}$-representation
if $\{\lambda_j\}_{j\ge0}\in\ell^p$, each $a_j$ is a
$(p,M,\varepsilon)_{\Delta_{\rm d}}$-molecule, and the series converges
in $\ell^2(\Zd)$. We define
$H^p_{\Delta_{\rm d},\mathrm{mol},M,\varepsilon}(\Zd)$ to be the
completion of the set
\[
 \bigl\{f\in\ell^2(\Zd):\ f\text{ has a molecular }
 (p,M,\varepsilon)_{\Delta_{\rm d}}\text{-representation}\bigr\}
\]
under the quasi-norm given by
\begin{equation}\label{eq:mol-norm}
 \|f\|_{H^p_{\Delta_{\rm d},\mathrm{mol},M,\varepsilon}(\Zd)}^p
 =\inf_{f=\sum_j\lambda_j a_j}\sum_j|\lambda_j|^p,
\end{equation}
where the infimum is taken over all molecular
$(p,M,\varepsilon)_{\Delta_{\rm d}}$-representations of $f$.
\end{definition}

By \cite[Theorem~3.3]{BD}, if $\varepsilon>0$ and
$M>\frac d2(\frac1p-\frac12)$, then
\[
 H^p_{\Delta_{\rm d}}(\Zd)
 =H^p_{\Delta_{\rm d},\mathrm{mol},M,\varepsilon}(\Zd)
\]
with equivalent quasi-norms. In particular, for a function in the
$\ell^2$, we have
\[
 \|f\|_{H^p_{\Delta_{\rm d}}(\Zd)}^p
 \simeq \inf\sum_j|\lambda_j|^p,
\]
where the infimum is taken over all molecular
$(p,M,\varepsilon)_{\Delta_{\rm d}}$-representations of $f$.

The following lemma shows that, in the definition of \(H^p_{\Delta_{\rm d} }(\mathbb Z^d)\) via molecules, we may restrict to balls of radius \(r>1\) without changing the space or its quasi-norm.

\begin{lemma}\label{lem:small-balls}
	In the molecular characterization of $H^p_{\D}(\Zd)$ one may restrict
	attention to balls of radius greater than $1$, without changing the space or
	its quasi-norm, up to a constant depending on $p,d,M,\varepsilon$.
\end{lemma}
\begin{proof}
	Let $a=\D^M b$  be a molecule associated with
	$B=B(n_0,r)$ where $0<r\le1$. Set $B'=B(n_0,2)$.
	Fix $0\le k\le M$ and write $q=1/2-1/p<0$. Then we have
	\[
	\|\D^k b\|_{\ell^2(S_j(B))}
	\le 2^{-j\varepsilon}r^{2(M-k)}|2^jB|^q.
	\]
	Since $|2^jB|\ge1$ and $r^{2(M-k)}\le1$, summing the pieces 
	which meet $B'$ yields
	$\|\D^k b\|_{\ell^2(B')}\lesssim |B'|^q 2^{2(M-k)}$.
	
	For $j\ge1$, we have
	\[
	\begin{aligned}
		\|\D^k b\|_{\ell^2(S_j(B'))}&\le \sum_{\ell: S_\ell(B)\cap S_j(B')\ne \emptyset }\|\D^k b\|_{\ell^2(S_\ell(B))} \\
		&\le \sum_{\ell: S_\ell(B)\cap S_j(B')\ne \emptyset } 2^{-\ell \varepsilon} r^{2(M-k)}|2^\ell B|^q.
	\end{aligned}
	\]
	Note that if $S_\ell(B)\cap S_j(B')\ne \emptyset$, then $2^\ell r\simeq 2^j$ and hence $|2^\ell B|\simeq |2^{j}B'|$; moreover, $\#\{\ell: S_\ell(B)\cap S_j(B')\ne \emptyset\}\lesssim 1 $ and the implicit constant is independent of $j$ and $r$. Consequently,
	\[
	\|\D^k b\|_{\ell^2(S_j(B'))}\lesssim 2^{-j\varepsilon}2^{2(M-k)}|2^jB'|^q,
	\]
	where in the last inequality we used the fact $r\le1$.
	
	Thus $a$ is, up to a constant independent of $r$, also a
	$(p,M,\varepsilon)$-molecule associated with $B'$.
	
	This completes our proof.
\end{proof}

We also recall the following complex interpolation result from
\cite[Proposition~3.6]{BD}.
\begin{proposition}\label{prop:interpolation}
Let $0<p_0,p_1<\infty$ and $0<\theta<1$. If
\[
 \frac1p=\frac{1-\theta}{p_0}+\frac\theta{p_1},
\]
then
\[
 [H^{p_0}_{\Delta_{\rm d}}(\Zd),H^{p_1}_{\Delta_{\rm d}}(\Zd)]_\theta
 =H^p_{\Delta_{\rm d}}(\Zd)
\]
with equivalent quasi-norms. Here $[\cdot,\cdot]_\theta$
denotes complex interpolation.
\end{proposition}

Finally, we recall the definition of the classical discrete Hardy
spaces. Let $0<p\le1$ and $1<q\le\infty$. A function $a$ on $\Zd$
is called a classical $(p,q)$-atom if there exists a ball $B\subset\Zd$
such that
\begin{enumerate}
 \item[(i)] $\operatorname{supp}a\subset B$;
 \item[(ii)] $\|a\|_{\ell^q(\Zd)}\le |B|^{1/q-1/p}$;
 \item[(iii)] $\displaystyle\sum_{n\in\Zd}n^\alpha a(n)=0$
 for every multi-index $\alpha\in\mathbb N_0^d$ satisfying
 $|\alpha|\le d(1/p-1)$.
\end{enumerate}
Here $n^\alpha=n_1^{\alpha_1}\cdots n_d^{\alpha_d}$, and
$1/q=0$ when $q=\infty$. The classical Hardy space
$H^p(\Zd)$ consists of distributions admitting a representation
\[
 f=\sum_{j=0}^\infty\lambda_j a_j,
 \qquad \sum_{j=0}^\infty|\lambda_j|^p<\infty,
\]
where each $a_j$ is a classical $(p,q)$-atom and the series converges
in $\mathcal S'(\Zd)$ (and in $\ell^1(\Zd)$ when $p=1$).
The quasi-norm is defined by
\[
 \|f\|_{H^p(\Zd)}^p
 =\inf\left\{\sum_{j=0}^\infty|\lambda_j|^p:
 f=\sum_{j=0}^\infty\lambda_j a_j
 \text{ is an atomic representation}\right\}.
\]
The space is independent of $q$ within the above range, up to
equivalent quasi-norms.

In \cite[Section~3.2, Theorem~3.14]{BDA}, it was shown that
$H^p_{\Delta_{\rm d}}(\mathbb Z)=H^p(\mathbb Z)$ for $0<p\le1$.
The same argument extends to $d\ge1$, using balls in $\Zd$
and multi-index polynomial moments in place of intervals and
one-variable moments. Thus
\[
 H^p_{\Delta_{\rm d}}(\Zd)=H^p(\Zd),\qquad 0<p\le1,
\]
with equivalent quasi-norms. 

\section{Proofs of main results}\label{sec:proofs}
\subsection{A localized Hardy-space estimate}\label{subsec:localized}
In this section we prove the estimate needed for the main theorem. We take $\rho=1+|\gamma|\ge1$.

Let $M\in\mathbb N$, $r\ge1$, and set
\begin{equation}\label{eq:h-def}
 h_{0,r}(\lambda)=(1-e^{-r^2\lambda})^M,\qquad
 h_{1,r}(\lambda)=(r^2\lambda)^M\bigl[1-(1-e^{-r^2\lambda})^M\bigr].
\end{equation}
Let $P(u)=1-(1-e^{-u})^M$. For every integer $j\ge0$, it is easy to see that for all $u>0$,
\begin{equation}\label{eq:cutoff-derivative}
  |(u\partial_u)^j(1-e^{-u})^M|
   +|(u\partial_u)^j(u^M P(u))|\le C_{M,j}\min\{1,u^M\}.
\end{equation}

In addition, let $\psi$ be the function in \eqref{eq-psi function}.
For every $R>0$ and $u>0$, we have
\begin{equation}\label{eq:derivative}
 \max_{0\le j\le N} 
 \left|\frac{d^j}{du^j}\big[\psi(u)u^{i\gamma}h_{\nu,r}(R^2u)\big]\right|
 \le C_{M,N}\rho^N\min\{1,(Rr)^{2M}\},\qquad \nu=0,1
\end{equation}
Indeed, $u\simeq1$ on $\supp\psi$, and each derivative of
$u^{i\gamma}$ costs at most one factor of $\rho$; the derivatives
of the other factors are controlled by \eqref{eq:cutoff-derivative}.
For the application of Proposition~\ref{prop:kernel} below, the
multiplier is
$G_{\ell,r}(\lambda)=\psi(R^{-2}\lambda)\lambda^{i\gamma}
h_{\nu,r}(\lambda)$. Equivalently, one may write
$F_{\ell,r}(\xi)=G_{\ell,r}(\sum_i\xi_i)$ for the joint
functional calculus; these two multipliers define the same operator.

\begin{lemma}\label{lem:localized}
Let $0<p\le1$, choose integers $N,M$ such that
\begin{equation}\label{eq:NM}
 N>d(1/p-1/2),\qquad 2M>N-d/2.
\end{equation}
Let $B=B(n_B,r_B)$ with $r_B\ge1$ and let $f$ be supported in $B$, with
$\norm{f}_2\le A|B|^{1/2-1/p}$. If $1\le r\le r_B$, then for $\nu=0,1$,
\begin{equation}\label{eq:local-lemma}
 \norm{S_\psi(\D^{i\gamma}h_{\nu,r}(\D)f)}_{\ell^p}
 \le C_{p,d,M,N}\,A\rho^{d(1/p-1/2)},
\end{equation}
where $S_\psi$ is the square function defined by \eqref{eq:square}.

\end{lemma}
\begin{proof}
Recall  $\rho=1+|\gamma|$. Set $B_\rho=B(n_B,\rho r_B)$ and $E_j=S_j(2B_\rho)$ for $j\ge1$.
Because $r_B\ge1$, we have
\begin{equation}\label{eq:ball-volumes}
 |B|\simeq r_B^d,\quad |2B_\rho|\simeq (\rho r_B)^d,\quad
 |E_j|\simeq (2^j\rho r_B)^d.
\end{equation}

We now write 
\[
\begin{aligned}
\norm{S_\psi(\D^{i\gamma}h_{\nu,r}(\D)f)}^p_{\ell^p}
&\le \norm{S_\psi(\D^{i\gamma}h_{\nu,r}(\D)f)}^p_{\ell^p(2B_\rho)}\\
&\quad+\norm{S_\psi(\D^{i\gamma}h_{\nu,r}(\D)f)}^p_{\ell^p(\Zd\setminus2B_\rho)}\\
&=: \textbf{H}_1+\textbf{H}_2.
\end{aligned}
\]

We first estimate the term $\textbf{H}_1 $. By the H\"older inequality,
\[
\textbf{H}_1\le |2B_\rho|^{1-p/2}
\|S_\psi(\D^{i\gamma}h_{\nu,r}(\D)f)\|_2^p
\]
Using the $L^2$-boundedness of $S_\psi$ and the spectral theorem, we further obtain 

\begin{align}\label{eq:local-part}
\textbf{H}_1
 &\less |2B_\rho|^{1-p/2}
 \|\D^{i\gamma}h_{\nu,r}(\D)f \|_2^p\notag\\
  &\less |2B_\rho|^{1-p/2}
 \|f \|_2^p\notag\\
 &\lesssim A^p(\rho r_B)^{d(1-p/2)}|B|^{p/2-1}
 \eqsim A^p\rho^{d(1-p/2)}.
\end{align}
\bigskip

It remains to estimate the term $\textbf{H}_2$. To do this,  we write 
\[
\begin{aligned}
	\textbf{H}_2&\le \sum_{j\ge 1}\norm{S_\psi(\D^{i\gamma}h_{\nu,r}(\D)f)}^p_{\ell^p(E_j)}\\
&\le \sum_{j\ge 1}\sum_{\ell\in\mathbb Z}\norm{\psi(2^{-2\ell}\D)(\D^{i\gamma}h_{\nu,r}(\D)f)}^p_{\ell^p(E_j)}\\
&=: \sum_{j\ge 1}\sum_{\ell\in\mathbb Z}\textbf{H}_{2,j,\ell}.
\end{aligned}
\]

Set $R=2^\ell$, $\ell\in\mathbb Z$ and
\[
T_{\ell,r}=\psi(R^{-2}\D)\D^{i\gamma}h_{\nu,r}(\D).
\]
The multiplier of $T_{\ell,r}$ is supported in $[R^2/4,4R^2]$.
Since the spectrum of $\D$ is contained in $[0,4d]$, we only need
$R\le 4\sqrt d$. Applying Proposition~\ref{prop:kernel} to
\[
F_{\ell,r}(\xi)=\psi\left(R^{-2}\sum_{i=1}^d\xi_i\right)
\left(\sum_{i=1}^d\xi_i\right)^{i\gamma}
h_{\nu,r}\left(\sum_{i=1}^d\xi_i\right),
\]
and using \eqref{eq:derivative}, we obtain
\begin{equation}\label{eq:kernel-loc}
	\sum_{n\in\Zd}|K_{T_{\ell,r}}(n,m)|^2(1+R|n-m|)^{2N}
	\lesssim R^d\rho^{2N}\min\{1,(Rr)^{4M}\},
\end{equation}
where the implicit constant is independent of   $\ell,r,\gamma$ and $m$.

For $n\in E_j$ and $m\in B$, $j\ge1$, one has
$|n-m|\gtrsim 2^j\rho r_B$. Write
$$
T_{\ell,r}f(n)=\sum_{m\in B}K_{T_{\ell,r}}(n,m)f(m).
$$
Using Minkowski's inequality in $\ell^2(E_j)$ and
\eqref{eq:kernel-loc}, we have
\begin{align*}
 \| T_{\ell,r}f\|_{\ell^2(E_j)}
 &\le \sum_{m\in B}|f(m)|
 \left(\sum_{n\in E_j}|K_{T_{\ell,r}}(n,m)|^2\right)^{1/2}\\
 &\lesssim R^{d/2}\rho^N
 \bigl(1+R2^j\rho r_B\bigr)^{-N}
 \min\{1,(Rr)^{2M}\}\sum_{m\in B}|f(m)|\\
 &\le R^{d/2}\rho^N
 \bigl(1+R2^j\rho r_B\bigr)^{-N}
 \min\{1,(Rr)^{2M}\}|B|^{1/2}\|f\|_2.
\end{align*}
Since $r\le r_B$, $\|f\|_2\le A|B|^{1/2-1/p}$, and
$1+R2^j\rho r_B\ge R2^j\rho r_B$, it follows that
\begin{equation}\label{eq:offdiagonal}
 \|T_{\ell,r}f\|_{\ell^2(E_j)}
 \lesssim AR^{d/2}(R2^jr_B)^{-N}
 \min\{1,(Rr_B)^{2M}\}|B|^{1-1/p}.
\end{equation}

This, together with the H\"older inequality on $E_j$,
\eqref{eq:ball-volumes}, and \eqref{eq:offdiagonal}, implies
\begin{align}\label{eq:annulus}
 \textbf{H}_{2,j,\ell}
 &\le |E_j|^{1-p/2}\| T_{\ell,r}f\|_{\ell^2(E_j)}^p\notag\\
 &\lesssim A^p(2^j\rho r_B)^{d(1-p/2)}
 R^{dp/2}(R2^jr_B)^{-Np}|B|^{p-1}
 \min\{1,(Rr_B)^{2Mp}\}\notag\\
 &\lesssim A^p\rho^{d(1-p/2)}
 2^{-jp[N-d(1/p-1/2)]}
 (Rr_B)^{-p(N-d/2)}\min\{1,(Rr_B)^{2Mp}\},
\end{align}
where in the last step we used $|B|\simeq r_B^d$ and $r_B^{d(1-p/2)}r_B^{d(p-1)}=r_B^{dp/2}$.

Thus \eqref{eq:annulus} yields
\begin{align}\label{eq:off-sums}
\textbf{H}_2 &\quad\lesssim A^p\rho^{d(1-p/2)}
 \sum_{j\ge1}2^{-jp[N-d(1/p-1/2)]}
 \sum_{\ell\in\mathbb Z}(2^\ell r_B)^{-p(N-d/2)}
 \min\{1,(2^\ell r_B)^{2Mp}\}.
\end{align}
The first series is convergent because $N>d(1/p-1/2)$.
For completeness, split the second series corresponding to
$2^\ell r_B\le1$ and $2^\ell r_B>1$. It is bounded by
\[
 \sum_{2^\ell r_B\le1}(2^\ell r_B)^{p[2M-N+d/2]}
 +\sum_{2^\ell r_B>1}(2^\ell r_B)^{-p(N-d/2)}\lesssim1.
\]
Both geometric series converge uniformly in $r_B$ due to \eqref{eq:NM}.
Combining \eqref{eq:local-part} and \eqref{eq:off-sums}, we have 
\[
 \|S_\psi(\D^{i\gamma}h_{\nu,r}(\D)f)\|_p^p
 \lesssim A^p\rho^{d(1-p/2)}.
\]

This completes our proof.
\end{proof}

\subsection{Upper estimates}
\begin{proof}[Proof of Theorem~\ref{thm:main}:]
Fix $0<p\le1$ and let $\rho=1+|\gamma|$.
Choose integers $N,M$ with
\[
 N>d\left(\frac1p-\frac12\right),\qquad
 2M>N-\frac d2,
\]
and take $M$ large enough for the molecular characterization
\eqref{eq:mol-norm}. By Lemma~\ref{lem:small-balls}, we may work with
molecules associated with balls of radius at least $1$.

Let $a=\D^M b$ be a normalized $(p,M,\varepsilon)$-molecule
associated with $B=B(n_B,r_B)$, where $r_B\ge1$. It suffices to prove that 
\begin{equation}\label{eq:single-mol}
 \|\D^{i\gamma}a\|_{H^p_{\D}}
 \lesssim \rho^{d(1/p-1/2)}.
\end{equation}
Write $P(u)=1-(1-e^{-u})^M$. The identity
\begin{equation}\label{eq:identity}
 I=(I-e^{-r_B^2\D})^M+P(r_B^2\D)
\end{equation}
holds on $\ell^2(\Zd)$ by the spectral theorem. Since
$a=\D^M b$, we obtain
\begin{equation}\label{eq:two-terms}
 \D^{i\gamma}a
 =\D^{i\gamma}h_{0,r_B}(\D)a
 +\D^{i\gamma}h_{1,r_B}(\D)(r_B^{-2M}b),
\end{equation}
where $h_{0,r},h_{1,r}$ are defined in \eqref{eq:h-def}.

For $k\ge0$, set
\[
 a_k=a\,1_{S_k(B)},\quad
 b_k=b\,1_{S_k(B)},\quad B_k=2^{k+1}B.
\]
These functions are supported in $B_k$. From \eqref{eq:mol} with
$k=M$ and $k=0$, respectively, we have
\begin{align*}
 \|a_k\|_2&\le2^{-k\varepsilon}|2^kB|^{1/2-1/p},\\
 r_B^{-2M}\|b_k\|_2&\le2^{-k\varepsilon}|2^kB|^{1/2-1/p}.
\end{align*}
Since $|2^kB|\simeq |B_k|$ and $B_k$ has radius
$R_k=2^{k+1}r_B\ge r_B\ge1$, these inequalities imply
\begin{equation}\label{eq:akbk}
 \|a_k\|_2+r_B^{-2M}\|b_k\|_2
 \lesssim 2^{-k\varepsilon}|B_k|^{1/2-1/p}.
\end{equation}

We apply Lemma~\ref{lem:localized}, first to $a_k$ with
$\nu=0$ and then to $r_B^{-2M}b_k$ with $\nu=1$.
In both applications the  ball is $B_k$, the auxiliary radius
is $r=r_B\le R_k$, and the normalization factor is
$A\lesssim2^{-k\varepsilon}$. Hence
\begin{align*}
 &\|S_\psi(\D^{i\gamma}h_{0,r_B}(\D)a_k)\|_p
 +\|S_\psi(\D^{i\gamma}h_{1,r_B}(\D)r_B^{-2M}b_k)\|_p \lesssim 2^{-k\varepsilon}\rho^{d(1/p-1/2)}.
\end{align*}

This, together with  \eqref{eq:two-terms}, implies
\[
\begin{aligned}
	\|S_\psi(\D^{i\gamma}a)\|_p^p&\lesssim \|S_\psi(\D^{i\gamma}h_{0,r_B}(\D)a)\|_p^p
+\|S_\psi(\D^{i\gamma}h_{1,r_B}(\D)r_B^{-2M}b)\|_p^p\\
&\lesssim \sum_{k\ge 0}\|S_\psi(\D^{i\gamma}h_{0,r_B}(\D)a_k)\|_p^p
+\sum_{k\ge 0}\|S_\psi(\D^{i\gamma}h_{1,r_B}(\D)r_B^{-2M}b_k)\|_p^p\\
	&\lesssim \rho^{d(1-p/2)}\sum_{k\ge0}2^{-kp\varepsilon}\\
	&\lesssim\rho^{d(1-p/2)}.
\end{aligned}
\]
This implies  \eqref{eq:single-mol} and hence completes our proof.

\end{proof}

We first prove the upper bound estimates in Theorem~\ref{thm:lp}. The sharpness will be given in the next subsection.
\begin{proof}[Proof of Theorem~\ref{thm:lp}]
For $p=2$, the spectral theorem gives
\[\|\D^{i\gamma}\|_{2\to2}=1.\]

For $1<p<2$, choose $p_0\in(0,1]$ with $p_0<p$.
By Theorem~\ref{thm:main}, $\D^{i\gamma}$ is bounded on
$H^{p_0}_{\D}$ with norm at most $C\rho^{d(1/p_0-1/2)}$.
Interpolating with $\ell^2=H^2_{\D}$ and using
\cite[Section~3]{BD}, we get
\[
 \norm{\D^{i\gamma}}_{\ell^p\to\ell^p}
 \lesssim\rho^{d(1/p-1/2)}.
\]
For $p>2$ the result follows by duality, since
$(\D^{i\gamma})^*=\D^{-i\gamma}$ on $\ell^2$.
\end{proof}

\subsection{Sharpness}
We next show that the exponents in Theorems~\ref{thm:main} and \ref{thm:lp} are optimal. We use only a frequency-localized stationary-phase argument.
\begin{lemma}\label{lem:stationary}
Let $\lambda$ be the function given in \eqref{eq:lambda}. There exist $\chi\in C_c^\infty((-\pi,\pi)^d)$, an open set
$U\subset\R^d$, and constants $c,\gamma_0>0$ such that the kernel
\[
 K_{\gamma,\chi}(n)=(2\pi)^{-d}\int_{[-\pi,\pi]^d}
 e^{i(n\cdot\theta+\gamma\log\lambda(\theta))}\chi(\theta)\,d\theta
\]
satisfies
\begin{equation}\label{eq:statbound}
 |K_{\gamma,\chi}(n)|\ge c|\gamma|^{-d/2}
\end{equation}
whenever $|\gamma|\ge\gamma_0$ and $n/\gamma\in U$. The set of such lattice points has cardinality at least $c|\gamma|^d$.
\end{lemma}

\begin{proof}
	Recall that
	\[
	\lambda(\theta)
	=2\sum_{j=1}^d(1-\cos\theta_j),
	\qquad \theta\in[-\pi,\pi]^d.
	\]
	We choose
	\[
	\theta_0=(\pi/2,0,\ldots,0).
	\]
	A direct calculation gives
	\[
	\nabla\log\lambda(\theta_0)=(1,0,\ldots,0)
	\]
	and
	\[
	D^2\log\lambda(\theta_0)
	=\operatorname{diag}(-1,1,\ldots,1).
	\]
	Hence,
	\[
	\det D^2\log\lambda(\theta_0)=-1\ne0.
	\]
	
	By the inverse function theorem, the map
	$\theta\mapsto\nabla\log\lambda(\theta)$
	is a diffeomorphism in a sufficiently small
	neighborhood $V$ of $\theta_0$.
	Choose $\chi\in C_c^\infty(V)$ such that
	$\chi\ge0$ and $\chi=1$ in a neighborhood
	of $\theta_0$.
	
	For $\gamma\ne0$, put $v=n/\gamma$ and write
	\[
	K_{\gamma,\chi}(n)
	=\frac1{(2\pi)^d}
	\int_{[-\pi,\pi]^d}
	e^{i\gamma\Phi_v(\theta)}
	\chi(\theta)\,d\theta,
	\]
	where
	\[
	\Phi_v(\theta)
	=v\cdot\theta+\log\lambda(\theta).
	\]
	Let $v_0=(-1,0,\ldots,0)$. Then
	\[
	\nabla\Phi_{v_0}(\theta_0)=0
	\]
	and
	\[
	\det D^2\Phi_{v_0}(\theta_0)=-1\ne0.
	\]
	
	By choosing a sufficiently small open
	neighborhood $U$ of $v_0$, we see that,
	for every $v\in U$, the phase $\Phi_v$
	has a unique critical point $\theta(v)$
	in $V$. Moreover, we may assume that
	\[
	\chi(\theta(v))=1
	\quad\text{and}\quad
	|\det D^2\Phi_v(\theta(v))|\simeq1
	\]
	uniformly for $v\in U$.
	
	Applying the stationary phase theorem,
	we obtain
	\[
	K_{\gamma,\chi}(n)
	=
	|\gamma|^{-d/2}
	e^{i\gamma\Phi_v(\theta(v))}
	\left(A_\gamma(v)+O(|\gamma|^{-1})\right),
	\]
	where the error is uniform for $v\in U$,
	and
	\[
	A_\gamma(v)
	=(2\pi)^{-d/2}
	e^{i\pi\operatorname{sgn}(\gamma)(d-2)/4}
	\frac{\chi(\theta(v))}
	{|\det D^2\Phi_v(\theta(v))|^{1/2}}.
	\]
	In particular, $|A_\gamma(v)|\ge c>0$
	uniformly for $v\in U$.
	
	It follows that, for all sufficiently
	large $|\gamma|$,
	\[
	|K_{\gamma,\chi}(n)|
	\ge c|\gamma|^{-d/2},
	\]
	whenever $n/\gamma\in U$.
	
	Finally, since $U$ is a nonempty open
	subset of $\mathbb R^d$, we have
	\[
	\#\{n\in\mathbb Z^d:n/\gamma\in U\}
	\ge c|\gamma|^d
	\]
	for all sufficiently large $|\gamma|$.
	This completes the proof.
\end{proof}

\begin{proof}[Proof of the sharpness in Theorem~\ref{thm:lp}]
Let $P_\chi$ be the smooth torus multiplier with symbol $\chi$, where $\chi$ is the function as in Lemma \ref{lem:stationary}. Its kernel $P_\chi\delta_0$ decays faster than any polynomial, so it belongs to $\ell^p(\Zd)$ for all $p>0$, and its norm is independent of $\gamma$. Since
\[
 \D^{i\gamma}P_\chi\delta_0=K_{\gamma,\chi},
\]
Lemma~\ref{lem:stationary} gives for $1<p<2$
\[
 \norm{\D^{i\gamma}P_\chi\delta_0}_p
 \gtrsim |\gamma|^{-d/2}(|\gamma|^d)^{1/p}
 =|\gamma|^{d(1/p-1/2)}.
\]
Dividing by $\norm{P_\chi\delta_0}_p$ proves the sharpness of \eqref{eq:Lp}. Duality gives the case $p>2$, and $p=2$ follows from the spectral theorem. For bounded $\gamma$, use $T_\gamma T_{-\gamma}=I$ and the equality $\|T_\gamma\|_{p\to p}=\|T_{-\gamma}\|_{p\to p}$, where $T_\gamma=\D^{i\gamma}$. Thus $\|T_\gamma\|_{p\to p}\ge1$ for every $\gamma$.

This completes our proof.
\end{proof}

\begin{proof}[Proof of sharpness in Theorem~\ref{thm:main}]
Let $q\in(1,2)$ and choose $0<\theta<1$ such that
\[
\frac1q=\frac{1-\theta}{1}+\frac\theta2.
\]
By Proposition~\ref{prop:interpolation},
$[H^p_{\D}(\Zd),\ell^2(\Zd)]_\theta=\ell^q(\Zd)$.
Since $\|\D^{i\gamma}\|_{\ell^2\to\ell^2}=1$, interpolation gives
\[
\|\D^{i\gamma}\|_{\ell^q\to\ell^q}
\lesssim
\|\D^{i\gamma}\|_{H^1_{\D}\to H^1_{\D}}^{1-\theta}.
\]
This proves the sharpness of \eqref{eq:Hp}.

This completes our proof.
\end{proof}

\subsection{The classical Calder\'on--Zygmund approach}\label{subsec:explicit}
We explain why the standard first-difference kernel estimates do not
recover the sharp power of $1+|\gamma|$. First we give an explicit
formula for the convolution kernel in dimension one.

\begin{proposition}\label{prop:gamma-kernel}
Let $d=1$ and $\gamma\in\R\setminus\{0\}$. For every $n\in\mathbb Z$,
\[
 K_\gamma(n)=(-1)^n\frac{\Gamma(1+2i\gamma)}
 {\Gamma(1+i\gamma+n)\Gamma(1+i\gamma-n)}.
\]
If $A_\gamma=(|\gamma|\tanh(\pi|\gamma|)/\pi)^{1/2}$, then
\[
 |K_\gamma(n)|=\frac{A_\gamma}{\sqrt{n^2+\gamma^2}},
\]
and, for $n\ge0$,
\[
 |K_\gamma(n+1)-K_\gamma(n)|
 =\frac{A_\gamma\sqrt{1+4\gamma^2}}
 {\sqrt{n^2+\gamma^2}\sqrt{(n+1)^2+\gamma^2}}.
\]
\end{proposition}
\begin{proof}
We first recall the kernel of $\Delta_1^{i\gamma}$.
For $f\in\ell^1(\mathbb Z)$, let
\[
 \widehat f(\theta)=\sum_{n\in\mathbb Z}f(n)e^{in\theta},
 \qquad \theta\in[-\pi,\pi].
\]
Since
\[
 \widehat{\Delta_1f}(\theta)
 =(2-2\cos\theta)\widehat f(\theta),
\]
the spectral theorem gives
\[
 \widehat{\Delta_1^{i\gamma}f}(\theta)
 =(2-2\cos\theta)^{i\gamma}\widehat f(\theta).
\]
Therefore, $\Delta_1^{i\gamma}$ is a convolution operator whose
kernel is
\[
 K_\gamma(n)=\frac1{2\pi}\int_{-\pi}^{\pi}
 (2-2\cos\theta)^{i\gamma}e^{-in\theta}\,d\theta,
 \qquad n\in\mathbb Z.
\]

For $\operatorname{Re}z>0$, we set
\[
 K_z(n)=\frac1{2\pi}\int_{-\pi}^{\pi}
 (2-2\cos\theta)^z e^{-in\theta}\,d\theta.
\]
Using $2-2\cos\theta=4\sin^2(\theta/2)$ and the change
of variable $\theta=2u$, we obtain
\[
 K_z(n)=\frac{4^z}{\pi}\int_0^\pi
 \sin^{2z}u\cos(2nu)\,du.
\]
In particular, by symmetry and the change of variable
$t=\sin^2u$ on $[0,\pi/2]$, we have
\begin{align*}
 K_z(0)
 &=\frac{4^z}{\pi}\int_0^\pi\sin^{2z}u\,du\\
 &=\frac{4^z}{\pi}\int_0^1
 t^{z-1/2}(1-t)^{-1/2}\,dt\\
 &=\frac{4^z}{\pi}B\left(z+\frac12,\frac12\right)\\
 &=\frac{4^z\Gamma(z+1/2)\Gamma(1/2)}
 {\pi\Gamma(z+1)}.
\end{align*}
Recall that $\Gamma(1/2)=\sqrt\pi$ and Legendre’s duplication formula gives
\[
 \Gamma(2z+1)=\frac{2^{2z}}{\sqrt\pi}
 \Gamma\left(z+\frac12\right)\Gamma(z+1).
\]
Consequently,
\[
 K_z(0)=\frac{\Gamma(1+2z)}{\Gamma(1+z)^2}.
\]

For $n\ge1$, let
$I_n=\int_0^\pi\sin^{2z}u\cos(2nu)\,du$.
Integrating the derivative of
$\sin^{2z+1}u\cos((2n-1)u)$ over $[0,\pi]$ gives
\[
 (n+z)I_n+(z-n+1)I_{n-1}=0.
\]
Thus
\[
 K_z(n)=\frac{n-1-z}{n+z}K_z(n-1).
\]
By iteration, and since $K_z(-n)=K_z(n)$, we obtain
\[
 K_z(n)=(-1)^n\frac{\Gamma(1+2z)}
 {\Gamma(1+z+n)\Gamma(1+z-n)},\qquad n\in\mathbb Z.
\]
The integral defining $K_z(n)$ is analytic for
$\operatorname{Re}z>-1/2$. Hence the identity also holds for
$z=i\gamma$, by analytic continuation. This gives the stated
formula for the kernel of $\Delta_1^{i\gamma}$.

Using the reflection formula, for $n\ge0$ we obtain
\[
 |K_\gamma(n)|
 =\frac{|\Gamma(1+2i\gamma)|\sinh(\pi|\gamma|)}{\pi}
   \frac{|\Gamma(n-i\gamma)|}{|\Gamma(n+1+i\gamma)|}
 =\frac{|\Gamma(1+2i\gamma)|\sinh(\pi|\gamma|)}
 {\pi\sqrt{n^2+\gamma^2}}.
\]
Since $|\Gamma(1+2i\gamma)|^2=2\pi|\gamma|/\sinh(2\pi|\gamma|)$,
this proves the formula for $|K_\gamma(n)|$.
Finally,
\[
 \frac{K_\gamma(n+1)}{K_\gamma(n)}
 =\frac{n-i\gamma}{n+1+i\gamma},
 \qquad
 K_\gamma(n+1)-K_\gamma(n)
 =-\frac{1+2i\gamma}{n+1+i\gamma}K_\gamma(n).
\]
Taking absolute values gives the stated difference formula.
\end{proof}

By Proposition~\ref{prop:gamma-kernel}, for $|\gamma|\ge1$,
the size coefficient has order $(1+|\gamma|)^{1/2}$.
However, the first-difference bound carries a factor of order
$(1+|\gamma|)^{3/2}$. Indeed, for
$n\ge2(1+|\gamma|)$ we have
\[
 |K_\gamma(n+1)-K_\gamma(n)|
 \simeq (1+|\gamma|)^{3/2}n^{-2}.
\]
Consequently, the usual first-difference Calder\'on--Zygmund
estimate does not give the sharp Hardy-space bound.

The same phenomenon occurs in higher dimensions. The low-frequency expansion $\lambda(\theta)=|\theta|^2+O(|\theta|^4)$ implies, for each fixed $\gamma\ne0$ and large $|n|$, that
\begin{equation}\label{eq:asymp}
	K_\gamma(n)=c_{d,\gamma}|n|^{-d-2i\gamma}+O_\gamma(|n|^{-d-2}),\quad
	c_{d,\gamma}=\frac{4^{i\gamma}\Gamma(d/2+i\gamma)}{\pi^{d/2}\Gamma(-i\gamma)}.
\end{equation}
Consequently, the leading term in a first difference has coefficient $(d+2i\gamma)c_{d,\gamma}$, of magnitude comparable to $|\gamma|^{d/2+1}$ for large $|\gamma|$. Hence the standard first-difference Calder\'on--Zygmund regularity estimate cannot have a constant of order $(1+|\gamma|)^{d/2}$ for all $\gamma$; compare \cite{Stein}. A first-difference Calder\'on--Zygmund argument therefore loses a factor
$1+|\gamma|$ at the endpoint. Interpolating the resulting bound
with $\ell^2$ gives the larger exponent $(d+2)|1/p-1/2|$.
This does not rule out more refined kernel methods, but explains why
ordinary first-difference estimates do not suffice. This explains the advantage of the localized weighted $\ell^2$ kernel argument used above.

\end{document}